\documentclass[a4paper, twoside]{scrartcl}
\usepackage{amsfonts,amssymb,amsmath,amsthm,amstext,amssymb,mathtools}
\usepackage{subcaption,float,color,graphicx,enumitem}
\usepackage{dsfont}  
\usepackage[normalem]{ulem} 
\usepackage[authoryear,sort,round]{natbib}  
\usepackage{fullpage} 
\usepackage{hyperref,nicefrac}

\usepackage[noabbrev,capitalise,nosort,nameinlink]{cleveref}

\crefname{assumption}{Assumption}{Assumptions}
\crefname{openproblem}{Open Problem}{Open Problem}
\crefname{notation}{Notation}{Notation}

\usepackage[table]{xcolor}

\usepackage{dsfont}

\usepackage{adjustbox}
\usepackage{tikz}
\usetikzlibrary{bayesnet,er,trees,shapes.symbols,mindmap,arrows,decorations,shapes.misc,shapes.arrows,chains,matrix,positioning,scopes,decorations.pathmorphing,patterns}
\usepackage{tikz-cd} 

\usepackage{dashbox}

\DeclareMathOperator{\epi}{epi}
\DeclareMathOperator{\CH}{Conv}
\DeclareMathOperator{\spann}{span}
\newcommand*{\cS}{\mathcal{S}}
\newcommand*{\cX}{\mathcal{X}}
\newcommand*{\cY}{\mathcal{Y}}
\newcommand*{\cZ}{\mathcal{Z}}

\newcommand*{\bR}{\mathbb{R}}
\newcommand*{\ebR}{\overline{\mathbb{R}}}
\newcommand*{\bN}{\mathbb{N}}

\newcommand*{\bE}{\mathbb{E}}

\newcommand*{\bP}{\mathbb{P}}

\newcommand*{\cE}{\mathcal{E}}
\newcommand*{\cF}{\mathcal{F}}
\newcommand*{\cH}{\mathcal{H}}

\newcommand*{\cG}{\mathcal{G}}

\makeatletter
\newcommand*\quark{\mathpalette\quark@{.5}}
\newcommand*\quark@[2]{\mathbin{\vcenter{\hbox{\scalebox{#2}{$\; \m@th#1\bullet \;$}}}}}
\makeatother

\makeatletter
\newcommand{\mylabel}[2]{#2\def\@currentlabel{#2}\label{#1}}
\makeatother

\newcommand*{\defterm}{\textit}

\DeclarePairedDelimiterX{\infdivx}[2]{(}{)}{#1\;\delimsize\|\;#2}

\definecolor{darkgreen}{rgb}{0,0.4,0}

\newcommand*{\todo}[1]{\bgroup\color{red}TODO:~#1\egroup}
\newcommand*{\pk}[1]{\bgroup\color{orange}HL:~#1\egroup}
\newcommand*{\ik}[1]{\bgroup\color{cyan}IK:~#1\egroup}

\DeclarePairedDelimiter\abs{\lvert}{\rvert}%
\DeclarePairedDelimiter\mynorm{\lVert}{\rVert}%
\makeatletter
\let\oldabs\abs
\def\abs{\@ifstar{\oldabs}{\oldabs*}}
\let\oldnorm\mynorm
\def\mynorm{\@ifstar{\oldnorm}{\oldnorm*}}
\makeatother

\theoremstyle{plain}
\newtheorem{theorem}{\sffamily Theorem}[section]

\newtheorem{lemma}[theorem]{\sffamily Lemma}
\newtheorem{corollary}[theorem]{\sffamily Corollary}
\theoremstyle{definition}
\newtheorem{definition}[theorem]{\sffamily Definition}

\newtheorem{example}[theorem]{\sffamily Example}

\newtheorem{remark}[theorem]{\sffamily Remark}

\newtheorem{assumption}[theorem]{\sffamily Assumption}

\newtheorem{problem}[theorem]{\sffamily Problem}

\newcommand{\norm}[1]{\lVert #1 \rVert}
\newcommand{\set}[2]{\{ #1 \mid #2 \}}

\newcommand{\Set}[2]{\left\{ #1 \,\middle\vert\, #2 \right\}}

\numberwithin{equation}{section}
\numberwithin{figure}{section}
\numberwithin{table}{section}

\usepackage{etex}
\usepackage[a4paper, top=88pt, bottom=88pt, left=72pt, right=72pt, headsep=16pt, footskip=28pt]{geometry}
\usepackage{hyperref}
\hypersetup{
	colorlinks,
	linkcolor=blue,
	citecolor=blue,
	urlcolor=blue,
}
\newcommand*{\arXiv}[1]{\bgroup\color{blue}\href{https://arxiv.org/abs/#1}{arXiv:#1}\egroup}
\newcommand*{\doi}[1]{\bgroup\color{blue}\href{https://doi.org/#1}{doi:#1}\egroup}
\newcommand*{\email}[1]{\bgroup\color{blue}\href{mailto:#1}{#1}\egroup}
\renewcommand*{\url}[1]{\bgroup\color{blue}\href{#1}{#1}\egroup}
\usepackage{enumitem, moreenum}
\setlist[enumerate]{nosep}
\setlist[itemize]{nosep}
\usepackage{mleftright} \mleftright
\renewcommand{\qedsymbol}{$\blacksquare$}
\renewenvironment{proof}[1][\proofname]{\noindent{\bfseries\sffamily #1.} }{\hfill\qedsymbol\medskip}

\usepackage[labelfont={sf,bf}]{caption}
\usepackage{scrlayer-scrpage, xhfill}
\automark[section]{section}
\setkomafont{pageheadfoot}{\normalcolor\sffamily}
\setkomafont{pagenumber}{\normalfont\normalsize\sffamily}
\clearpairofpagestyles
\let\oldtitle\title
\renewcommand{\title}[1]{\oldtitle{#1}\newcommand{\theshorttitle}{#1}}
\newcommand{\shorttitle}[1]{\renewcommand{\theshorttitle}{#1}}
\let\oldauthor\author
\renewcommand{\author}[1]{\oldauthor{#1}\newcommand{\theshortauthor}{#1}}
\newcommand{\shortauthor}[1]{\renewcommand{\theshortauthor}{#1}}
\cohead{\xrfill[0.525ex]{0.6pt}~\theshorttitle~\xrfill[0.525ex]{0.6pt}}
\cehead{\xrfill[0.525ex]{0.6pt}~\theshortauthor~\xrfill[0.525ex]{0.6pt}}
\newcommand{\theabstract}[1]{\par\bgroup\noindent\textbf{\textsf{Abstract.}} #1\egroup}
\newcommand{\thekeywords}[1]{\par\smallskip\bgroup\noindent\textbf{\textsf{Keywords.}}\newcommand{\and}{ $\bullet$ } #1\egroup}
\newcommand{\themsc}[1]{\par\smallskip\bgroup\noindent\textbf{\textsf{2020 Mathematics Subject Classification.}}\newcommand{\and}{ $\bullet$ } #1\egroup}

\newcommand*{\affilref}[1]{\ref{affiliation#1}}
\newcommand*{\affiliation}[3]{
	\footnotetext[#1]{\label{affiliation#2}#3}
}

\setlist{topsep=0.3ex, itemsep=0.3ex}

\title{Graph Convex Hull Bounds\\ as generalized Jensen Inequalities}
\shorttitle{Graph Convex Hull Bounds as generalized Jensen Inequalities}
\author{%
	Ilja~Klebanov\textsuperscript{\affilref{FUB}}%
}
\shortauthor{I.~Klebanov}
\date{\today}

\begin{document}
\maketitle
\affiliation{1}{FUB}{Freie Universit{\"a}t Berlin, Arnimallee 6, 14195 Berlin, Germany (\email{klebanov@zedat.fu-berlin.de})}

\begin{abstract}\small
	\theabstract{Jensen's inequality is ubiquitous in measure and probability theory, statistics, machine learning, information theory and many other areas of mathematics and data science.
It states that, for any convex function $f\colon K \to \mathbb{R}$ defined on a convex domain $K \subseteq \mathbb{R}^{d}$ and any random variable $X$ taking values in $K$, $\mathbb{E}[f(X)] \geq f(\mathbb{E}[X])$.
In this paper, sharp upper and lower bounds on $\mathbb{E}[f(X)]$, termed ``graph convex hull bounds'', are derived for arbitrary functions $f$ on arbitrary domains $K$, thereby extensively generalizing Jensen's inequality.
The derivation of these bounds necessitates the investigation of the convex hull of the graph of $f$, which can be challenging for complex functions.
On the other hand, once these inequalities are established, they hold, just like Jensen's inequality, for \emph{any} $K$-valued random variable $X$.
Therefore, these bounds are of particular interest in cases where $f$ is relatively simple and $X$ is complicated or unknown.
Both finite- and infinite-dimensional domains and codomains of $f$ are covered as well as analogous bounds for conditional expectations and Markov operators.}
	\thekeywords{{Jensen's inequality}%
\and%
{Convex Hull}%
\and%
{non-convex functions}%
\and%
{Markov operators}%
\and%
{conditional expectation}%
\and%
{Hahn--Banach separation theorem}
}
	\themsc{26D15  
\and
28B05  
\and
46A55  
\and
52A40  
\and
60B11  
\and
37A30  

}
\end{abstract}


\section{Introduction}
\label{section:Introduction}

In many theoretical and practical derivations, it is challenging to exactly compute the mean $\bE[f(X)]$ of a function $f\colon K \to \bR$, where $K \subseteq \bR^{d}$ and $d\in\bN$, applied to a $K$-valued random variable $X$.
Thus, one has to rely on bounds on $\bE[f(X)]$.
This is particularly the case when $X$ is only known in terms of its mean $\bE[X]$, requiring dependence solely on the knowledge of $\bE[X]$ and $f$, leading to the following general problem addressed in this paper:
\begin{problem}
\label{problem:E_fX_bounds}
Establish (sharp) bounds on the mean $\bE[f(X)]$ given $\bE[X]$ and $f$.
\end{problem}
In the finite-dimensional setup described above,
Jensen's inequality \citep{Jensen1906convexes} offers partial answers for convex functions $f$ on convex domains $K$, asserting that
\begin{equation}
\label{equ:standard_Jensen}
\bE[f(X)] \geq f(\bE[X])
\end{equation}
(without specifying an upper bound).
While Jensen's inequality is a foundational results in measure theory and is widely utilized in probability theory, statistics, information theory, statistical physics, and many other research areas, it exhibits several crucial limitations and restrictions.
The aim of this paper is to derive new bounds on $\bE[f(X)]$, termed \emph{graph convex hull bounds}, that address these shortcomings, in particular:
\begin{itemize}
\item 
Jensen's inequality says nothing about functions $f$ that are neither convex nor concave, while the graph convex hull bounds hold for arbitrary functions.
\item 
A significant limitation of Jensen's inequality is its requirement for the function $f$ to be defined over a \emph{convex} domain $K$.
Our graph convex hull bounds, on the other hand, are free from such constraints on the domain of $f$, allowing for even disconnected domains as illustrated in \Cref{ex:convex_hull_2} and \Cref{fig:Jensen_generalized}, enhancing their versatility.
\item 
Jensen's inequality does not offer any form of ``reverse'' bounds of the forms $\bE[f(X)] \leq s + f(\bE[X])$ or $\bE[f(X)] \leq c f(\bE[X])$ for some $s = s(X) \geq 0$, $c = c(X) \geq 1$, a feature that could be invaluable in various mathematical and applied contexts and can certainly be obtained for specific functions or function classes \citep{Wunder2021reverse,Dragomir2013reverse,Khan2020converses,Dragomir2001reverse}.
The minimal value $s = s(X) \geq 0$ which satisfies above property is called the \emph{Jensen gap} \citep{Abramovich2016gap,Ullah2021gap,Khan2020gap}.
The graph convex hull bounds provide sharp bounds on $\bE[f(X)]$ both from above and below.
However, unlike the Jensen gap, our bounds depend on $X$ only through $\bE[X]$ (cf.\ \Cref{problem:E_fX_bounds}), hence they are more generally applicable but can only be sharp in a weaker sense.
\item
For domains $K$ within infinite-dimensional topological vector spaces $\cX$, Jensen's inequality requires $f$ to meet certain continuity conditions \citep{Perlman1974Jensen}, a stipulation not needed by our graph convex hull bounds.
\end{itemize}

The graph convex hull bounds are obtained by exploiting the basic fact that the mean of the pair $(X,f(X))$ lies in the closure $\overline{\CH(\cG(f))}$ of the convex hull of the graph $\cG(f)$ of $f$, cf.\ \Cref{cor:mean_in_convex_hull_of_graph} and \Cref{fig:Jensen_generalized} below.
This can be shown by a simple application of the Hahn--Banach separation theorem
--- 
characterize $\overline{\CH(\cG(f))}$ as an intersection of half-spaces and note that the expected value $\bE$ as an operator ``respects'' each half-space, since it is linear, positive and $\bE[Z] = Z$ for constant $Z$.
While the graph convex hull bounds in concrete applications require some investigation of $\overline{\CH(\cG(f))}$ (note that Jensen's inequality also requires some analysis of $f$, namely the verification of its convexity), all inequalities are derived solely from the properties of $f$ and hold, just like Jensen's inequality, for \emph{any} random variable $X$ no matter its probability distribution.
This property is crucial because, in practice, the function $f$ is often sufficiently simple to analyze, at least numerically, whereas the distribution of $X$ might be complex, unknown, or one requires a priori bounds that hold \emph{uniformly} for all random variables taking values in $K$.

Note that the three properties of $\bE$ mentioned above for the derivation of $(\bE[X],\bE[f(X)]) \in \overline{\CH(\cG(f))}$ are of algebraic nature and have little to do with measure or integration theory.
Hence, the graph convex hull bounds can be extended to a broader class of linear operators beyond expected values, so-called Markov operators, as well as to conditional expectations, which are well-known facts in the case of Jensen's inequality, see e.g.\ \citep[Equation~(1.2.1)]{Bakry2014Markov} and \citep[Theorem~10.2.7]{Dudley2002real}.
Once these three versions of the graph convex hull bounds are proven, the corresponding versions of Jensen's inequality follow as simple corollaries, providing a novel and considerably simpler proof of this famous result (cf.\ \Cref{rem:our_derivation_of_Jensen_is_simpler}).
By working in a more general setup than is typical for Markov operators (\Cref{rem:def_of_Markov_operators_typical_assumptions}), we also strongly generalize the corresponding version of Jensen's inequality.

The paper is structured as follows.
After laying out the general setup and notation in \Cref{section:Setup},
the graph convex hull bounds are derived for expected values (\Cref{section:Results}), Markov operators (\Cref{section:Markov_Operators}), and conditional expectations (\Cref{section:Conditional_Results}).
The corresponding versions of Jensen's inequality follow as \Cref{cor:Jensen_follows_from_our_bounds,cor:Jensen_follows_from_our_bounds_Markov,cor:Jensen_follows_from_our_bounds_conditional} in their respective sections.

\section{Preliminaries}
\label{section:Setup}

In order to formulate the results within the most general setting possible, we assume that $\cX$ and $\cY$ are both real, Hausdorff, and locally convex topological vector spaces and $f \colon K\to \cY$, where $K\subseteq \cX$ is an arbitrary domain.
However, readers are encouraged to conceptualize $\cX = \bR^{d}$ and $\cY = \bR$ for simplicity.
Furthermore, $(\Omega,\Sigma,\bP)$ will denote a probability space, and $X \colon \Omega \to \cX$ will represent a function such that $(X,f(X))$ is weakly integrable\footnote{As is typical for random variables $X$, we write $f(X)$ for the composition $f\circ X$; and extend this notation to maps $X$ that are not measurable.}.
Therefore, all expected values are interpreted in the weak sense, which implies that $X$ does not necessarily need to be measurable, that is, a random variable.
Recall that a map $X \colon \Omega \to \cX$ is called \defterm{weakly integrable} if there exists $e\in\cX$ such that, for any $\ell \in \cX'$, $\ell(X) \in L^{1}(\bP)$ and $\ell(e) = \int_{\Omega} \ell(X)\, \mathrm{d} \bP$, in which case the expected value of $X$ is defined by $\bE[X] \coloneqq \int_{\Omega} X\, \mathrm{d} \bP \coloneqq e$ \citep[Definition~3.26]{Rudin1991functional}.
Here, $\cX'$ denotes the continuous dual of $\cX$.
In particular, each Pettis or Bochner integrable map is weakly integrable, and readers may consider Bochner integrable random variables $X$ for simplicity.
When working with Markov operators in \Cref{section:Markov_Operators}, we will not impose any measurability or integrability assumptions on $(X,f(X))$, while for conditional expectations in \Cref{section:Conditional_Results} we will require $\cX$ and $\cY$ to be Banach spaces and $(X,f(X))$ to be Bochner integrable.
Recall that a random variable $X\colon (\Omega,\Sigma,\bP) \to\cX$ is Bochner or strongly integrable if and only if $X$ is measurable and $\int_{\cX} \norm{X}_{\cX} \, \mathrm{d} \bP < \infty$ \citep[Theorem~II.2.2]{Diestel1977vectormeasures}, in which case we write $X \in L^{1}(\bP;\cX)$.
The relevant assumptions will be detailed separately in each section, namely in \Cref{assump:general,assump:Markov_operators,assump:conditional}.

Throughout the paper, we use the following notation.
We denote the closure of  a subset $A$ of a real topological vector space by $\overline{A}$ and the convex hull of $A$ by
\begin{equation}
\label{equ:def_convex_hull}
\CH(A) 
\coloneqq	
\bigcap \set{ C \supset A }{C \text{ convex}}
=
\Set{\textstyle \sum_{j=1}^{J} w_{j} a_{j}}{J\in\bN,\, a_{j} \in A,\, w_{j} \geq 0,\, \textstyle \sum_{j=1}^{J} w_{j} = 1}.
\end{equation}
Further, $\cG(f) = \set{(x,f(x))}{x\in K}$ denotes the graph of a function $f \colon K \to \cY$ and, for $M \subseteq \cX \times \cY$ and $x\in\cX$,
\[
M|_x \coloneqq \set{y\in \cY}{(x,y)\in M}
\]
(note that $\overline{M}|_x$ and $\overline{M|_x}$ might not coincide).
Finally, while we clearly have in mind inequalities with respect to some total or partial order or preorder on $\cY$, most statements are presented for arbitrary binary relations on $\cY$, which we continue to denote by $\leq$ for readability reasons.
However, in the case where $\cY = \bR$, we consistently assume the canonical total order.
In certain situations, we will extend the space $\cY$ by adding the elements $\pm \infty$, which are presumed to satisfy $-\infty \leq y \leq \infty$ for each $y\in\cY$, and denote $\overline{\cY} \coloneqq \cY \cup \{ \pm \infty \}$, as is typical for the extended real number line $\ebR = [-\infty,\infty] = \bR \cup \{ \pm \infty \}$.
Closed intervals in $\overline{\cY}$ are defined by
$
[a,b]
\coloneqq
\set{y\in\overline{\cY}}{a \leq y \leq b},
\ 
a,b \in \overline{\cY}.
$
We now state a simple preliminary lemma that we will use to derive Jensen's inequality from the graph convex hull bounds:
\begin{lemma}
\label{lemma:for_Jensen}
Let $f\colon K \to \bR$ be a convex function on a convex domain $K\subseteq \bR^{d}$.
Then its epigraph $\epi(f) \coloneqq \set{(x,y)\in K\times \bR}{y \geq f(x)} \supseteq \cG(f)$ is convex and, consequently, $\CH(\cG(f))|_{x} \subseteq [f(x),\infty)$ for each $x \in K$.
If $K = \bR^{d}$, then $\overline{\CH(\cG(f))}|_{x} \subseteq [f(x),\infty)$ for each $x \in K$.
\end{lemma}
\begin{proof}
The convexity of $\epi(f)$ is a well-known result \citep[Theorem~4.1]{Rockafellar1997ConvexAnalysis}.
It follows that $\CH(\cG(f)) \subseteq \epi(f)$ by \eqref{equ:def_convex_hull} and $\CH(\cG(f))|_{x} \subseteq \epi(f)|_{x} = [f(x),\infty)$ for each $x \in K$.
If $K = \bR^{d}$, then $f$ is continuous \citep[Theorem~4.1.3]{Borwein2006convex} and $\epi(f)$ is closed.
Hence, $\overline{\CH(\cG(f))}|_{x} \subseteq \overline{\epi(f)}|_{x} = \epi(f)|_{x} = [f(x),\infty)$ for each $x \in K$.
\end{proof}

\section{Graph Convex Hull Bounds for Expected Values}
\label{section:Results}

Throughout this section, we make the following general assumptions:
\begin{assumption}
	\label{assump:general}
	$(\Omega,\Sigma,\bP)$ is a probability space, $\cX,\cY,\cZ$ are real, Hausdorff, and locally convex topological vector spaces (equipped with their Borel $\sigma$-algebras).
	Furthermore, $\leq$ represents any binary relation on $\cY$, with the stipulation that it defaults to the canonical total order in the case $\cY = \bR$ (hence the notation).
	Finally, $K \subseteq \cX$, $f\colon K \to \cY$ and $X \colon \Omega \to K$ such that the pair $(X,f(X))$ is weakly integrable.
\end{assumption}
The following theorem presents a well-known and intuitive result that the mean $\bE[Z]$ of a random variable $Z$ taking values in a closed convex subset $C$ of
$\cZ$ lies within this set\footnote{Note that this addresses a distinct issue from that explored by Choquet theory, and the Krein--Milman theorem in particular, as discussed in \citep{Phelps2001Choquet}, where a bounded, closed convex set is analyzed in relation to the convex hull of its \emph{extreme points}.
Extreme points will be of no relevance to this paper.}: $\bE[Z] \in C$.
If $\cZ = \bR^{d}$, the closedness assumption is not required.
Its proof utilizes the Hahn--Banach separation theorem.

\begin{theorem}[{mean/barycenter of convex set lies within its closure}]
\label{thm:mean_in_convex_hull}
Let \Cref{assump:general} hold and let $Z\colon \Omega \to A$ be a weakly integrable map taking values in a subset $A\subseteq \cZ$.
Then $\bE[Z] \in \CH(A)$ if $\cZ = \bR^{d}$, $d\in\bN$, and $\bE[Z] \in \overline{\CH(A)}$ in general\footnote{For an example of a random variable $Z$ taking values in a convex subset $C$ of an infinite-dimensional space which satisfies $\bE[Z] \in \overline{C}\setminus C$, see \citep[Remark~3.2]{Perlman1974Jensen}.}.
Further, for any $z \in \CH(A)$ there exist a probability space 
$(\tilde{\Omega},\tilde{\Sigma},\tilde{\bP})$ and a weakly integrable map $\tilde{Z}\colon \tilde{\Omega} \to A$ such that $\bE[\tilde{Z}] = z$.
\end{theorem}
\begin{proof}
See \citep[Theorem~3.1]{Perlman1974Jensen} for the general statement $\bE[Z] \in \overline{\CH(A)}$ and \citep[Theorem~10.2.6.]{Dudley2002real} for the proof of $\bE[Z] \in \CH(A)$ in the case $\cZ = \bR^{d}$ (though our setup is slightly more general, the proof goes similarly).
Now let $z \in \CH(A)$. Then it can be written as a convex combination of elements of $A$, i.e.\ 
$
z = \sum_{j=1}^{J} w_{j} a_{j},
$
for some $J\in\bN$, $a_{j} \in A$, $w_{j} \geq 0$, $j=1,\dots,J$, with $\sum_{j=1}^{J} w_{j} = 1$.
Choose $\tilde{\Omega} = \{ 1,\dots,J \}$ with distribution $\tilde{\bP}$ given by the probability vector $(w_{1},\dots,w_{J})$ and let $\tilde{Z}(j) = a_{j}$.
Then $\bE[\tilde{Z}] = \sum_{j=1}^{J} w_{j} a_{j} = z$.
\end{proof}

An application of above theorem to the convex hull of the graph of a function $f$ yields:

\begin{corollary}
\label{cor:mean_in_convex_hull_of_graph}
Under \Cref{assump:general}, $\bE[(X,f(X))] \in \CH(\cG(f))$ if $\cX = \bR^{m},\, \cY = \bR^{n},\ m,n\in\bN$, and $\bE[(X,f(X))]
\in
\overline{\CH(\cG(f))}$ in general.
Further, for any $(x,y) \in \CH(\cG(f))$ there exist a probability space $(\tilde{\Omega},\tilde{\Sigma},\tilde{\bP})$ and $\tilde{X}\colon \tilde{\Omega} \to K$ such that the pair $(\tilde{X},f(\tilde{X}))$ is weakly integrable and $\bE[(\tilde{X},f(\tilde{X}))] = (x,y)$.
\end{corollary}
\begin{proof}
This follows from \Cref{thm:mean_in_convex_hull} for $Z \coloneqq (X,f(X))$ and $A \coloneqq \cG(f) \subseteq \cX \times \cY \eqqcolon \cZ$ (note that any map $\tilde{Z}$ taking values in $\cG(f)$ is of the form $(\tilde{X},f(\tilde{X}))$).
\end{proof}

The following result is basically a restatement of \Cref{cor:mean_in_convex_hull_of_graph}.
Since it is the main result of this section, we state it as a theorem:

\begin{theorem}[{graph convex hull bounds on $\bE[f(X)]$}]
\label{thm:main_result}
Under \Cref{assump:general}, let the convex hull of the graph of $f$ be ``enclosed'' by two ``envelope'' functions $g_{l},g_{u} \colon \overline{\CH(K)} \to \overline{\cY}$ in the following way:
\begin{align}
\label{equ:envelope_characterization_convex_hull}
\begin{split}
\overline{\CH(\cG(f))}
&\subseteq
\set{(x,y) \in \overline{\CH(K)} \times \cY}{g_{l}(x)\leq y \leq g_{u}(x)},
\\[1ex]
\text{i.e.}
\qquad
\overline{\CH(\cG(f))}|_{x}
&\subseteq
[g_{l}(x),g_{u}(x)]
\qquad
\text{for each }
x\in \overline{\CH(K)}.
\end{split}
\end{align}
Then
\begin{equation}
\label{equ:main_result}
g_{l}(\bE[X])
\leq
\bE[f(X)]
\leq
g_{u}(\bE[X]).
\end{equation}
If $\cX = \bR^{m},\, \cY = \bR^{n},\ m,n\in\bN$, then $\overline{\CH(K)}$ and $\overline{\CH(\cG(f))}$ can be replaced by $\CH(K)$ and $\CH(\cG(f))$, respectively.
Further, for $\cY = \bR$, \eqref{equ:main_result} is sharp in the following sense:
\begin{itemize}
\item 
If $\cX = \bR^{m}$ and $g_{l}(x) = \inf (\CH(\cG(f))|_{x})$ and $g_{u}(x) = \sup (\CH(\cG(f))|_{x})$ for each $x\in \CH(K)$, then, for every $x \in \CH(K)$ and $y \in (g_{l}(x),g_{u}(x))$, there exist a probability space $(\tilde{\Omega},\tilde{\Sigma},\tilde{\bP})$ and $\tilde{X}\colon \tilde{\Omega} \to K$ such that the pair $(\tilde{X},f(\tilde{X}))$ is weakly integrable and $(\bE[\tilde{X}] , \bE[f(\tilde{X})]) = (x,y)$.
\item 
In general, if $g_{l}(x) = \inf (\overline{\CH(\cG(f))}|_{x})$ and $g_{u}(x) = \sup (\overline{\CH(\cG(f))}|_{x})$ for each $x\in \overline{\CH(K)}$, then, for every $x \in \overline{\CH(K)}$ and $y \in (g_{l}(x),g_{u}(x))$ and for any neighborhood $U$ of $(x,y)$, there exist a probability space $(\tilde{\Omega},\tilde{\Sigma},\tilde{\bP})$ and $\tilde{X}\colon \tilde{\Omega} \to K$ such that the pair $(\tilde{X},f(\tilde{X}))$ is weakly integrable and $(\bE[\tilde{X}], \bE[f(\tilde{X})]) \in U$.
\end{itemize}

\end{theorem}
\begin{proof}
Since $\bE[X] \in \overline{\CH(K)}$ and $(\bE[X],\bE[f(X)]) \in \overline{\CH(\cG(f))}$ by \Cref{thm:mean_in_convex_hull,cor:mean_in_convex_hull_of_graph},
\[
\bE[f(X)] \in \overline{\CH(\cG(f))}|_{\bE[X]} \subseteq [g_{l}(\bE[X]),g_{u}(\bE[X])],
\]
cf.\ \Cref{fig:Jensen_generalized}.
In the finite-dimensional case $\cX = \bR^{m},\, \cY = \bR^{n}$, we can omit all the closures by \Cref{thm:mean_in_convex_hull,cor:mean_in_convex_hull_of_graph}.
The claim on the sharpness of the bounds also follows directly from \Cref{cor:mean_in_convex_hull_of_graph}.
\end{proof}

\begin{remark}
\label{rem:Comments_on_main_result}
We emphasize again that, just like Jensen's inequality, the bounds $g_{l}$ and $g_{u}$ are determined solely by the function $f$ and hold for \emph{any} $K$-valued map $X$ that meets the minimal integrability assumptions.
\end{remark}

A direct application of the graph convex hull bounds from \Cref{thm:main_result} yields Jensen's inequality:

\begin{corollary}[Jensen's inequality]
	\label{cor:Jensen_follows_from_our_bounds}
	Let \Cref{assump:general} hold and $f\colon K \to \bR$ be a convex function on a convex domain $K\subseteq \bR^{d}$.
	Then $\bE[f(X)] \geq f(\bE[X])$.
\end{corollary}
\begin{proof}
	This follows directly from \Cref{lemma:for_Jensen} and \Cref{thm:main_result} with $g_l \coloneqq f$.
\end{proof}

\begin{remark}
	\label{rem:our_derivation_of_Jensen_is_simpler}
	It's worth noting that this novel proof strategy of Jensen's inequality is significantly simpler than traditional approaches.
	In the proof of, for example, \citep[Theorem~10.2.6.]{Dudley2002real}, the author outlines:
	\begin{itemize}
		\item
		a first argument showing $\bE[X]\in K$, akin to \Cref{thm:mean_in_convex_hull}, ensuring $f(\bE[X])$ is well defined,
		\item
		a second straightforward argument similar to \Cref{lemma:for_Jensen} and
		\item 
		a third, more complex argument to establish the inequality itself.
	\end{itemize}
	Yet, this latter argument can be entirely bypassed by applying the first argument a second time, this time to the pair $(X,f(X))$, resulting in $(\bE[X],\bE[f(X)]) \in \CH(\cG(f))$ and thus directly proving Jensen's inequality.
	This elegant approach appears to have been previously overlooked, as well as its application to non-convex functions and domains as demonstrated in \Cref{thm:main_result}.
\end{remark}

Note that the lower bound $f(\bE[X])$ on $\bE[f(X)]$ in Jensen's inequality \eqref{equ:standard_Jensen} is superfluous in our setup and may not even be well defined (since $\bE[X]$ is not guaranteed to lie in $K$, cf.\ \Cref{ex:convex_hull_2}).
However, if comparing $\bE[f(X)]$ and $f(\bE[X])$ is crucial, the stated result leads to the following consequences whenever $f(\bE[X])$ exists:
\newpage
\begin{corollary}
\label{cor:further_Jensen_type_inequalities}
Let the assumptions of \Cref{thm:main_result} hold and let $\cY = \bR$.
Then
\vspace{1ex}
\begin{enumerate}[label = (\alph*)]
	\itemsep=2ex	
	\item
	\label{item:Jensen_type_inequality_2}
	$c_{l}(\bE[X]) \, f(\bE[X])
	\leq
	\bE[f(X)]
	\leq
	c_{u}(\bE[X]) \, f(\bE[X])$,	
	\item
	\label{item:Jensen_type_inequality_3}
	$
	\max\big(\inf_{x\in K} f(x) \, ,\, f(\bE[X]) - s_{u}\big)
	\leq
	\bE[f(X)]
	\leq
	\min\big(\sup_{x\in K} f(x) \, ,\, f(\bE[X]) + s_{l}\big),$
	\item
	\label{item:Jensen_type_inequality_4}
	$\hat{c}_{l} \, f(\bE[X])
	\leq
	\bE[f(X)]
	\leq
	\hat{c}_{u} \, f(\bE[X])$,
\end{enumerate}
\vspace{1ex}
whenever the quantities
\begin{align*}
	c_{l}(x) &\coloneqq \tfrac{g_{l}(x)}{f(x)},
	&
	s_{l} &\coloneqq\norm{f-g_{l}}_{\infty},
	&
	\hat{c}_{l} &\coloneqq \inf_{x\in \overline{\CH(K)}}\tfrac{g_{l}(x)}{f(x)},
	\\
	c_{u}(x) &\coloneqq \tfrac{g_{u}(x)}{f(x)},
	&
	s_{u} &\coloneqq \norm{f-g_{u}}_{\infty},
	&
	\hat{c}_{u} &\coloneqq \sup_{x\in \overline{\CH(K)}}\tfrac{g_{u}(x)}{f(x)},
\end{align*}
as well as $f(\bE[X])$ are well defined.
If $\cX = \bR^{m},\, \cY = \bR^{n},\ m,n\in\bN$, $\overline{\CH(K)}$ can be replaced by $\CH(K)$ in the definition of $\hat{c}_{l}$ and $\hat{c}_{u}$.
\end{corollary}

\begin{proof}
\ref{item:Jensen_type_inequality_2} and \ref{item:Jensen_type_inequality_3} follow directly from \Cref{thm:main_result}, while \ref{item:Jensen_type_inequality_4} follows from \ref{item:Jensen_type_inequality_2}.	
\end{proof}

\begin{remark}
\label{rem:Comments_on_further_Jensen_type_inequalities}
While the bounds in \Cref{thm:main_result} and \Cref{cor:further_Jensen_type_inequalities}~\ref{item:Jensen_type_inequality_2} require the knowledge of $\bE[X]$ (cf.\ \Cref{problem:E_fX_bounds}), statements~\ref{item:Jensen_type_inequality_3} and~\ref{item:Jensen_type_inequality_4} compare $\bE[f(X)]$ with $f(\bE[X])$ without relying on this information.
In particular, \ref{item:Jensen_type_inequality_3} and~\ref{item:Jensen_type_inequality_4} offer a priori bounds applicable uniformly across all random variables $X$ with values in $K$.
Consequently, in transitioning from~\ref{item:Jensen_type_inequality_2} to~\ref{item:Jensen_type_inequality_4}, the lower bound $c_{l}(\bE[X])$ must be substituted with its most conservative estimate, namely, the infimum of $c_{l}(x)$ across all potential values of $x$ (and the supremum in the case of the upper bound).
However, these bounds are still sharper than the obvious inequalities (whenever well defined)
\[
\inf_{x\in K} f(x)
\leq
\bE[f(X)]
\leq
\sup_{x\in K} f(x),
\qquad
f(\bE[X]) \inf_{x,y\in \overline{K}} \tfrac{f(x)}{f(y)}
\leq
\bE[f(X)]
\leq
f(\bE[X]) \sup_{x,y\in \overline{K}} \tfrac{f(x)}{f(y)}.
\]
\end{remark}

\begin{figure}[t]
	\centering
	\begin{subfigure}[b]{0.48\textwidth}
		\centering
		\includegraphics[width=\textwidth]{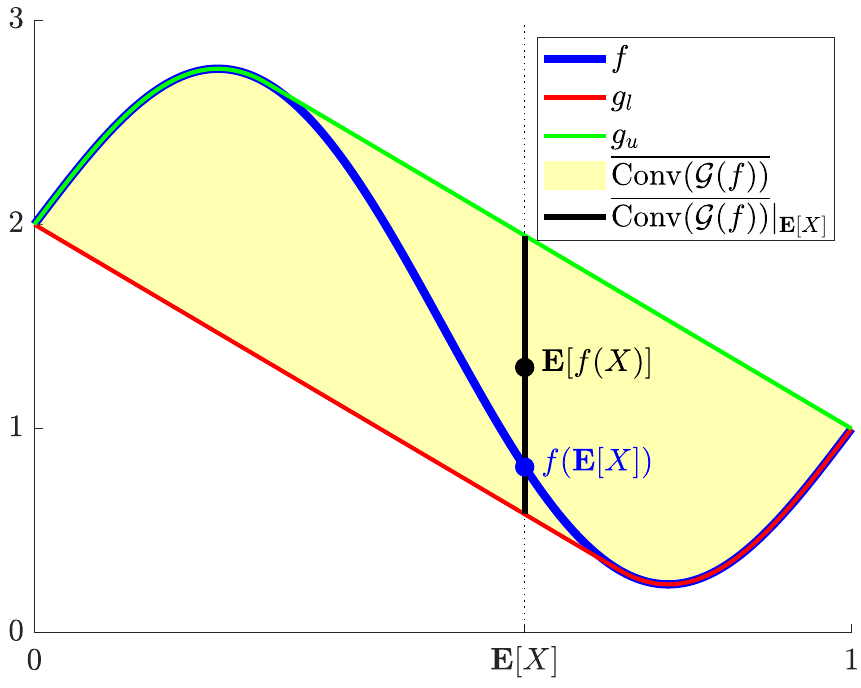}
	\end{subfigure}
	\hspace{1em}
	\begin{subfigure}[b]{0.48\textwidth}
		\centering
		\includegraphics[width=\textwidth]{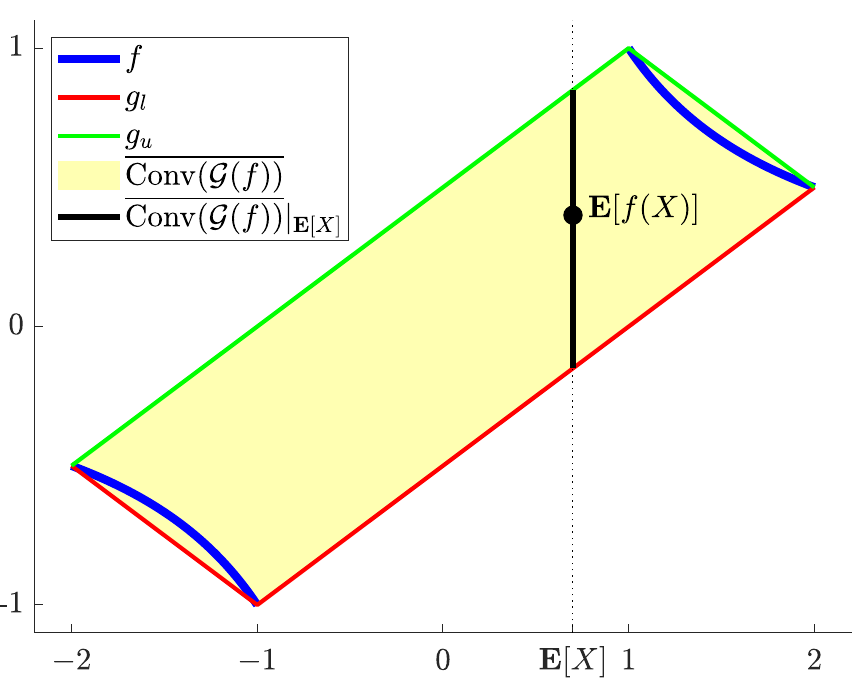}
	\end{subfigure}
	\caption{
		Illustration of \Cref{ex:convex_hull_1,ex:convex_hull_2}.
		Since the mean $\bE[(X,f(X))] = (\bE[X],\bE[f(X)])$ (black dot) of the pair $(X,f(X))$ lies in the closure $\overline{\CH(\cG(f))}$ of the convex hull of the graph of $f$, the value $\bE[f(X)]$ is restricted to the thick black line $\overline{\CH}_{\bE[X]}$, providing the bounds in \Cref{thm:main_result}.
		Note that, in the second example, the domain $K = [-2,-1] \cup [1,2]$ of $f$ is disconnected and $f(\bE[X])$ is not even defined whenever $\bE[X] \in (-1,1)$, which does not affect \Cref{thm:main_result}.
	}
	\label{fig:Jensen_generalized}
\end{figure}

\begin{example}
	\label{ex:convex_hull_1}
	Let $K = [0,1]$ and $f\colon K \to \bR$, $f(x) = 2 - x + \sin(2 \pi x)$, visualized in \Cref{fig:Jensen_generalized} (left), which is neither convex nor concave.
	Then the ``envelope'' functions of $\CH(\cG(f))$ are given by
	\[
	g_{l}(x)
	=
	\begin{cases}
		2-ax & \text{if } x \leq x_{\ast},\\
		f(x) & \text{otherwise},
	\end{cases}
	\qquad
	g_{u}(x)
	=
	\begin{cases}
		f(x)   & \text{if } x \leq 1 - x_{\ast},\\
		1+a-ax & \text{otherwise},
	\end{cases}
	\]
	where $x_{\ast} \approx 0.715$ is the solution of $2\pi\cos(2\pi x) \, x = \sin(2\pi x)$ and $a = 1 - 2 \pi \cos(2 \pi x_{\ast})$.	
	Note that the constants $\hat{c}_{l} \approx 0.5$ and $\hat{c}_{u}\approx 6.5$ from \Cref{cor:further_Jensen_type_inequalities} compare favorably with the obvious bounds $\inf_{x,y\in K}\tfrac{f(x)}{f(y)} \approx 0.09$ and $\sup_{x,y\in K}\tfrac{f(x)}{f(y)} \approx 11.6.$ from \Cref{rem:Comments_on_main_result} (it is easy to see how the function $f$ can be modified to make the improvements over these bounds as large as desired).	
\end{example}

\begin{example}
	\label{ex:convex_hull_2}
	Let $K = [-2,-1] \cup [1,2]$ be the \emph{disconnected} domain of the function $f\colon K \to \bR$, $f(x) = x^{-1}$, visualized in \Cref{fig:Jensen_generalized} (right), which again is neither convex nor concave.
	Again, the ``envelope'' functions of $\CH(\cG(f))$ from \Cref{thm:main_result}, given by the piecewise linear functions
	\[
	g_{l}(x)
	=
	\begin{cases}
		(-3-x)/2 & \text{if } x \in [-2,-1],\\
		(x-1)/2  & \text{if } x \in [-1,2],
	\end{cases}
	\qquad
	g_{u}(x)
	=
	\begin{cases}
		(x+1)/2 & \text{if } x \in [-2,1],\\
		(3-x)/2 & \text{if } x \in [1,2],
	\end{cases}
	\]
	provide sharp bounds on $\bE[f(X)]$.
	In this case, it might be impossible to compare $\bE[f(X)]$ with $f(\bE[X])$, since the latter is not defined whenever $\bE[X] \in (-1,1)$.
\end{example}

\begin{remark}
While in the above examples the envelope functions $g_{l}$ and $g_{u}$ were determined analytically (up to the approximation of $x_{\ast}$), they usually can be computed numerically with high precision and little computational effort even for a complicated function $f$, as long as its domain $K$ is bounded.
\end{remark}

\section{Graph Convex Hull Bounds for Markov Operators}
\label{section:Markov_Operators}

Note that the only properties of the expected value $\bE$ used to prove $\bE[Z] \in \overline{\CH(A)}$ in \Cref{thm:mean_in_convex_hull} are its linearity, the positivity of $\ell(Z)$ for each $\ell \in \cZ'$ (meaning that $\ell(Z)\geq 0$ implies $\bE[\ell(Z)]\geq 0$), and the property that $\bE[Z] = Z$ for each constant random variable $Z \in \cZ$ (and the same is true for Jensen's inequality).
Therefore, Jensen's inequality and the graph convex hull bounds can be interpreted as algebraic results rather than as measure theoretic ones, and can be extended to operators that do not necessarily involve integration, but merely satisfy the three properties above.
This generalization is the purpose of this section and will also enable us to cover conditional expectations in the subsequent section.
Such operators are known as \emph{Markov operators} and play a crucial role in the analysis of time evolution phenomena and dynamical systems \citep{Bakry2014Markov}.
They are well-known to satisfy Jensen's inequality \citep[Equation~(1.2.1)]{Bakry2014Markov}, yet we significantly broaden the setting in which it applies, cf.\ \Cref{rem:def_of_Markov_operators_typical_assumptions}.

\begin{definition}
\label{def:Markov_operators}
Let $\Omega_{1},\Omega_{2}$ be sets and 
$\cS_{1},\cS_{2}$ be vector spaces of functions on $\Omega_{1}$ and $\Omega_{2}$ that contain the unit constant functions $\mathds{1}_{\Omega_{1}}\colon \Omega_{1} \to \bR$, $t \mapsto 1$ and $\mathds{1}_{\Omega_{2}}$, respectively,
e.g.\ $\cS_{1} = L^{1}(\bP_{1})$, if $(\Omega_{1},\Sigma_{1},\bP_{1})$ is a probability space.
A linear operator $\cE \colon \cS_{1} \to \cS_{2}$ is called a \defterm{Markov operator} if
\begin{enumerate}[label = (\roman*)]
	\item
	$\cE(\mathds{1}_{\Omega_{1}}) = \mathds{1}_{\Omega_{2}}$,
	\item 
	$\cE(\varphi) \geq 0$ for each $\varphi\in \cS_{1}$ with $\varphi \geq 0$.
\end{enumerate}
Further, let $\cZ$ be a real topological vector space and
\[
\cS^{\cZ}_{1}
\subseteq
\set{Z_{1} \colon \Omega_{1} \to \cZ}{\ell \in \cZ'\colon \ell(Z_{1}) \in \cS_{1}},
\quad
\cS^{\cZ}_{2}
\subseteq
\set{Z_{2} \colon \Omega_{2} \to \cZ}{\ell \in \cZ'\colon \ell(Z_{2}) \in \cS_{2}},
\]
e.g.\ $\cS^{\cZ}_{1} = L^{1}(\bP_{1};\cZ)$ in the case $\cS_{1} = L^{1}(\bP_{1})$.
We call $\cE^{\cZ} \colon \cS^{\cZ}_{1} \to \cS^{\cZ}_{2}$ a \defterm{$\cZ$-valued Markov operator} associated with the Markov operator $\cE$ if $\ell(\cE^{\cZ}(Z_{1})) = \cE(\ell(Z_{1}))$ for each $\ell \in \cZ'$ and $Z_{1} \in \cS^{\cZ}_{1}$, i.e.\  if the following diagram commutes for each $\ell \in \cZ'$.
\begin{center}
\begin{tikzcd}[column sep=4em,row sep=2em]
	\cS^{\cZ}_{1}
	\arrow[r,rightarrow,"\cE^{\cZ}"]
	\arrow[d,rightarrow,"\ell \circ"]
	&
	\cS^{\cZ}_{2}
	\arrow[d,rightarrow,"\ell \circ"]
	\\
	\cS_{1}
	\arrow[r,rightarrow,"\cE"]
	&
	\cS_{2}
\end{tikzcd}
\end{center}
\end{definition}

\begin{remark}
\label{rem:def_of_Markov_operators_typical_assumptions}
Typically, $\Omega_{1},\Omega_{2}$ are assumed to be measurable spaces or measure spaces, and the elements in $\cS_{1},\cS_{2}$ to be measurable or to lie in $L^{1}$ with respect to certain measures \citep{Bakry2014Markov,eisner2015operator}.
Further, some authors require Markov operators to be endomorphisms, i.e.\ $\cS_{1} = \cS_{2}$ \citep{Bakry2014Markov}, to be $L^{1}$-contractions \citep{Lin1979Markov,Rudnicki2000Markov} or to preserve $L^{1}$-norm \citep{eisner2015operator}.
By working under the minimal requirements in \Cref{def:Markov_operators} and by introducing the entirely new notion of a $\cZ$-valued Markov operator, we thus significantly expand the scope of the currently known version of Jensen's inequality for Markov operators \citep[Equation~(1.2.1)]{Bakry2014Markov}.
\end{remark}

\begin{example}[{cf.\ \Cref{section:Results}}]
	\label{ex:mean_as_Markov_operator}
	Let $(\Omega_{1},\Sigma_{1},\bP_{1})=(\Omega_{2},\Sigma_{2},\bP_{2})$ be a probability space and $\cS_{1} = L^{1}(\bP_{1})$ and $\cS_{2} = \spann (\mathds{1}_{\Omega_{2}})$.
	Further, let $\cZ$ be a real, Hausdorff, locally convex topological vector space and $\cS^{\cZ}_{1} = \set{Z_{1} \colon \Omega_{1} \to \cZ}{Z_{1} \text{ weakly integrable}}$,
	$\cS^{\cZ}_{2} = \set{Z_{2} \colon \Omega_{2} \to \cZ}{Z_{2} \text{ constant}}$.
	Then the expected value $\bE\colon \cS^{\cZ}_{1} \to \cS^{\cZ}_{2}$ is a $\cZ$-valued Markov operator associated with the Markov operator $\bE\colon \cS_{1} \to \cS_{2}$, where $\cS^{\cZ}_{2}$ is identified with $\cZ$ and $\cS_{2}$ with $\bR$.
	Note that the same notation $\bE$ is used for both operators as is the standard in probability theory.
\end{example}

\begin{example}[{cf.\ \Cref{section:Conditional_Results}}]
	\label{ex:conditional_mean_as_Markov_operator}
	Let $(\Omega_{1},\Sigma_{1},\bP_{1})=(\Omega_{2},\Sigma_{2},\bP_{2})$ be a probability space and $\cF \subseteq \Sigma_{1}$ be a sub-$\sigma$-algebra of $\Sigma_{1}$.
	Let $\cS_{1} = L^{1}(\Sigma_{1},\bP_{1})$ and $\cS_{2} = L^{1}(\cF,\bP_{1}|_{\cF})$.
	Further, let $\cZ$ be a Banach space, $\cS^{\cZ}_{1} = L^{1}(\Sigma_{1},\bP_{1};\cZ)$ and
	$\cS^{\cZ}_{2} = L^{1}(\cF,\bP|_{\cF};\cZ)$.
	Then the conditional expectation $\bE[\quark | \cF]\colon \cS^{\cZ}_{1} \to \cS^{\cZ}_{2}$ is a $\cZ$-valued Markov operator associated with the Markov operator $\bE[\quark | \cF]\colon \cS_{1} \to \cS_{2}$ (this follows easily from e.g.\ \cite[Theorem~II.2.6]{Diestel1977vectormeasures}).
	Again, the same notation $\bE[\quark | \cF]$ is used for both operators as is common practice.
\end{example}

To extend the findings in \Cref{section:Results} to Markov operators, we establish the subsequent assumptions for this section:

\begin{assumption}
	\label{assump:Markov_operators}
	$\cX,\cY$ are two real, Hausdorff, and locally convex topological vector spaces and $\Omega_{1},\, \Omega_{2},\, \cS_{1},\, \cS_{2},\, \cS^{\cX}_{1},\, \cS^{\cX}_{2},\, \cS^{\cY}_{1},\, \cS^{\cY}_{2}$ are as in \Cref{def:Markov_operators}, with $\cZ$ replaced by $\cX$ and $\cY$, respectively.
	Further, $\cE^{\cX} \colon \cS^{\cX}_{1} \to \cS^{\cX}_{2}$ and $\cE^{\cY} \colon \cS^{\cY}_{1} \to \cS^{\cY}_{2}$ are $\cX$-valued and $\cY$-valued Markov operators, respectively, associated with the \emph{same} Markov operator $\cE \colon \cS_{1} \to \cS_{2}$.
	$\leq$ is any binary relation on $\cY$, which we assume to be the canonical total order in the case $\cY = \bR$.
	Finally $K \subseteq \cX$, $f\colon K \to \cY$ and $X \colon \Omega_{1} \to K$ such that the pair $(X,f(X)) \in \cS^{\cX}_{1}\times \cS^{\cY}_{1}$. 
\end{assumption}

\begin{lemma}
	\label{lemma:pair_of_Markov_operators}
	Under \Cref{assump:Markov_operators}, the operator $\cE^{\cX \times \cY} \colon \cS^{\cX}_{1}\times \cS^{\cY}_{1} \to \cS^{\cX}_{2}\times \cS^{\cY}_{2},\ (\tilde{X},\tilde{Y}) \mapsto (\cE^{\cX}(\tilde{X}),\cE^{\cY}(\tilde{Y}))$ is an $(\cX\times\cY)$-valued Markov operator associated with the same Markov operator $\cE$.
\end{lemma}
\begin{proof}
	Let $\tilde{Z} = (\tilde{X},\tilde{Y}) \in \cS^{\cX}_{1}\times \cS^{\cY}_{1}$ and $\ell \in (\cX\times\cY)'$.
	Then $\ell$ is of the form $\ell(x,y) = \ell^{\cX}(x)+\ell^{\cY}(y)$ for some $\ell^{\cX} \in \cX',\, \ell^{\cY} \in \cY'$.
	Hence, $\ell(\tilde{Z}) = \ell^{\cX}(\tilde{X}) + \ell^{\cY}(\tilde{Y}) \in \cS$ (since $\cS$ is a vector space), as required (and similarly for $\cS^{\cX}_{2}\times \cS^{\cY}_{2}$).
	Further,
	\[
	\ell(\cE^{\cX \times \cY}(\tilde{Z}))
	=
	\ell^{\cX} (\cE^{\cX}(\tilde{X}) + \ell^{\cY} (\cE^{\cY}(\tilde{Y})
	=
	\cE(\ell^{\cX}(\tilde{X})) + \cE (\ell^{\cY} (\tilde{Y}))
	=
	\cE(\ell(\tilde{Z})).
	\]
\end{proof}

Let us now formulate the analogue of \Cref{thm:mean_in_convex_hull} for Markov operators:

\begin{theorem}[{Markov operators respect closed convex set constraints}]
\label{thm:Markov_operators_respect_convex_hull}
Let $\Omega_{1},\, \Omega_{2},\, \cS_{1},\, \cS_{2},\, \cE$, $\cZ,\ \cS^{\cZ}_{1},\, \cS^{\cZ}_{2},\, \cE^{\cZ}$ be as in \Cref{def:Markov_operators} and assume that $\cZ$ is Hausdorff and locally convex.
Further, let $Z\in \cS_{1}$ with $Z(\Omega_{1}) \subseteq A$ for some subset $A\subseteq \cZ$.
Then $(\cE^{\cZ}(Z))(\Omega_{2}) \subseteq \overline{\CH(A)}$.
\end{theorem}
\begin{proof}
By the Hahn--Banach separation theorem, since $\overline{\CH(A)}$ is convex, it coincides with the intersection of all closed half-spaces $\cH_{\ell, s} \coloneqq \set{ z \in \cZ}{\ell(z)\geq s} \supseteq \overline{\CH(A)}$ including it, where $\ell \in \cZ'\setminus \{ 0 \}$ and $s \in \bR$.
Since, for $\omega \in \Omega$,
\begin{align*}
Z(\Omega_{1}) \subseteq \cH_{\ell,s}
&\implies
\ell(Z) \geq s
\\
&\implies
\cE(\ell(Z)) \geq s
\\
&\implies
\ell(\cE^{\cZ}(Z) \geq s
\\
&\implies
\cE^{\cZ}(Z)(\Omega_{2}) \subseteq \cH_{\ell,s},
\end{align*}
$(\cE^{\cZ}(Z))(\Omega_{2})$ lies in each half-space $\cH_{\ell, s}$ including $\overline{\CH(A)}$, proving the claim.
\end{proof}

An application of above theorem to the convex hull of the graph of a function $f$ yields:

\begin{corollary}
\label{cor:Markov_operators_respect_convex_hull_of_graph}
Under \Cref{assump:conditional},
$(\cE^{\cX}(X), \cE^{\cY}(f(X))) (\Omega_{2}) \subseteq \overline{\CH(\cG(f))}$.
\end{corollary}
\begin{proof}
Using \Cref{lemma:pair_of_Markov_operators}, this follows from \Cref{thm:Markov_operators_respect_convex_hull} for $Z \coloneqq (X,f(X))$ and $A \coloneqq \cG(f) \subseteq \cX \times \cY \eqqcolon \cZ$.
\end{proof}

As in \Cref{section:Results}, a reformulation of above corollary results in graph convex hull bounds for Markov operators in place of expected values:

\begin{theorem}[{graph convex hull bounds on $\cE^{\cY}(f(X))$}]
\label{thm:main_result_Markov_operators}
Under \Cref{assump:conditional}, let the convex hull of the graph of $f$ be ``enclosed'' by two ``envelope'' functions $g_{l},g_{u} \colon \overline{\CH(K)} \to \overline{\cY}$ satisfying \eqref{equ:envelope_characterization_convex_hull}.
Then
\begin{equation}
\label{equ:main_result_Markov_operators}
g_{l}(\cE^{\cX}(X))
\leq
\cE^{\cY}(f(X))
\leq
g_{u}(\cE^{\cX}(X)).
\end{equation}
Note that, in contrast to \Cref{thm:main_result}, all objects in \eqref{equ:main_result_Markov_operators} are functions on $\Omega_{2}$.
\end{theorem}
\begin{proof}
Since $(\cE^{\cX}(X))(\Omega_{2}) \subseteq \overline{\CH(K)}$ and $( \cE^{\cX}(X),\cE^{\cY}(f(X)) )(\Omega_{2}) \subseteq \overline{\CH(\cG(f))}$ by \Cref{thm:Markov_operators_respect_convex_hull,cor:Markov_operators_respect_convex_hull_of_graph},
\[
(\cE^{\cY}(f(X)))(\Omega_{2})
\subseteq
\overline{\CH(\cG(f))}|_{\cE^{\cX}(X)} 
\subseteq
[g_{l}(\cE^{\cX}(X)),g_{u}(\cE^{\cX}(X))].
\]
\end{proof}

\begin{corollary}
\label{cor:further_Markov_operators_Jensen_type_inequalities}
Let the assumptions of \Cref{thm:main_result_Markov_operators} hold and let $\cY = \bR$.
Then, using the notation from \Cref{cor:further_Jensen_type_inequalities},
\vspace{1ex}
\begin{enumerate}[label = (\alph*)]
	\itemsep=2ex	
	\item
	\label{item:Markov_operators_Jensen_type_inequality_2}
	$c_{l}(\cE^{\cX}(X)) \, f(\cE^{\cX}(X))
	\leq
	\cE^{\cY}(f(X))
	\leq
	c_{u}(\cE^{\cX}(X)) \, f(\cE^{\cX}(X))$,
	\item
	\label{item:Markov_operators_Jensen_type_inequality_3}
	$
	\max\big(\inf_{x\in K} f(x) \, ,\, f(\cE^{\cX}(X)) - s_{u}\big)
	\leq
	\cE^{\cY}(f(X))
	\leq
	\min\big(\sup_{x\in K} f(x) \, ,\, f(\cE^{\cX}(X)) + s_{l}\big),$
	\item
	\label{item:Markov_operators_Jensen_type_inequality_4}
	$\hat{c}_{l} \, f(\cE^{\cX}(X))
	\leq
	\cE^{\cY}(f(X))
	\leq
	\hat{c}_{u} \, f(\cE^{\cX}(X))$,
\end{enumerate}
\vspace{1ex}
whenever these quantities are well defined.
\end{corollary}

\begin{proof}
\ref{item:Markov_operators_Jensen_type_inequality_2} and \ref{item:Markov_operators_Jensen_type_inequality_3} follow directly from \Cref{thm:main_result_Markov_operators}, while \ref{item:Markov_operators_Jensen_type_inequality_4} follows from \ref{item:Markov_operators_Jensen_type_inequality_2}.	
\end{proof}

We conclude this section by deriving Jensen's inequality as a straightforward corollary of \Cref{thm:main_result_Markov_operators}. To the best of the author's knowledge, Jensen's inequality for Markov operators has not been established with this level of generality, cf.\ \Cref{rem:def_of_Markov_operators_typical_assumptions}.

\begin{corollary}[Jensen's inequality for Markov operators]
	\label{cor:Jensen_follows_from_our_bounds_Markov}
	Let \Cref{assump:Markov_operators} hold and $f\colon \bR^{d} \to \bR$ be a convex function.
	Then $\cE^{\cY}(f(X)) \geq f(\cE^{\cX}(X))$.
\end{corollary}
\begin{proof}
	This follows directly from \Cref{thm:main_result_Markov_operators} and \Cref{lemma:for_Jensen}.
\end{proof}

\section{Graph Convex Hull Bounds for Conditional Expectations}
\label{section:Conditional_Results}

It is well-known that Jensen's inequality also holds for \emph{conditional expectations} \citep[Theorem~10.2.7]{Dudley2002real}.
In this section, we establish graph convex hull bounds for conditional expectations, representing a specific instance of the broader framework outlined in \Cref{section:Markov_Operators} (cf.\ \Cref{ex:conditional_mean_as_Markov_operator}) and deriving directly from the results therein.
Since conditional expectations are commonly defined over Banach spaces and for Bochner integrable random variables \citep[Section~V.1]{Diestel1977vectormeasures}, this will require slightly stronger assumptions:

\begin{assumption}
	\label{assump:conditional}
	$(\Omega,\Sigma,\bP)$ is a probability space, $\cX,\cY,\cZ$ are real Banach spaces and $K \subseteq \cX$ a subset of $\cX$.	
	Further, $\leq$ is any binary relation on $\cY$, which we assume to be the canonical total order in the case $\cY = \bR$ and $\cF \subseteq \Sigma$ is a sub-$\sigma$-algebra with $\bE[\quark | \cF]$ denoting the corresponding conditional expectation.
	Finally, $K \subseteq \cX$, $f\colon K \to \bR$ and $X \colon \Omega \to K$ is a random variable such that the pair $(X,f(X))$ is Bochner integrable.
\end{assumption}

\begin{theorem}[{conditional expectations respect closed convex set constraints}]
\label{thm:conditional_mean_in_convex_hull}
Let $Z \colon (\Omega,\Sigma,\bP) \to A$ be a Bochner integrable random variable taking values in a subset $A\subseteq \cZ$ of a real Banach space $\cZ$ and let $\cF \subseteq \Sigma$ be a sub-$\sigma$-algebra.
Then $\bE[Z | \cF] \in \overline{\CH(A)}$ $\bP$-almost surely.
\end{theorem}
\begin{proof}
Note that, as is typical for conditional expectations, no difference is made in the notation between $\bE[\quark | \cF] \colon L^{1}(\Sigma,\bP) \to L^{1}(\cF,\bP|_{\cF})$ and $\bE[\quark | \cF] \colon L^{1}(\Sigma,\bP;\cZ) \to L^{1}(\cF,\bP|_{\cF};\cZ)$ (corresponding to $\cE$ and $\cE_{\cZ}$ in \Cref{def:Markov_operators}).
Since $\bE[\ell(Z)| \cF] = \ell( \bE[Z | \cF] )$ for each $\ell \in \cZ'$ (this follows easily from e.g.\ \cite[Theorem~II.2.6]{Diestel1977vectormeasures}), and that $W\geq s$ implies $\bE[W | \cF]\geq s$ $\bP$-almost surely for each real-valued random variable $W \colon (\Omega,\Sigma,\bP) \to \bR$ by \cite[Theorem~8.1(ii)]{Kallenberg2021foundations}, the claim follows directly from \Cref{thm:Markov_operators_respect_convex_hull}.
\end{proof}
\begin{remark}
In contrast to \Cref{section:Markov_Operators}, all statements in $L^{1}$ can only hold $\bP$-almost surely.
Hence, when drawing the final conclusion from the proof of \Cref{thm:Markov_operators_respect_convex_hull} to the one of \Cref{thm:conditional_mean_in_convex_hull}, we must exercise caution when considering the intersection over all (possibly uncountably many) half-spaces that contain $\overline{\CH(A)}$.
However, following the construction of conditional expectations \citep[Theorem~V.1.4]{Diestel1977vectormeasures}, there exists a representative of $\bE[Z|\cF]$ which lies in all these half-spaces simultaneously.
\end{remark}

This theorem's application to the convex hull of the graph of a function $f$ yields:

\begin{corollary}
\label{cor:conditional_mean_in_convex_hull_of_graph}
Under \Cref{assump:conditional}, $\bE[(X,f(X))| \cF]
\in
\overline{\CH(\cG(f))}$ $\bP$-almost surely.
\end{corollary}
\begin{proof}
This follows from \Cref{thm:conditional_mean_in_convex_hull} for $Z \coloneqq (X,f(X))$ and $A \coloneqq \cG(f) \subseteq \cX \times \cY \eqqcolon \cZ$.
\end{proof}

\begin{theorem}[{graph convex hull bounds on $\bE[f(X)| \cF]$}]
\label{thm:main_result_conditional}
Under \Cref{assump:conditional}, let the convex hull of the graph of $f$ be ``enclosed'' by two ``envelope'' functions $g_{l},g_{u} \colon \overline{\CH(K)} \to \overline{\cY}$ satisfying \eqref{equ:envelope_characterization_convex_hull}.
Then
\begin{equation}
\label{equ:main_result_conditional}
g_{l}(\bE[X | \cF])
\leq
\bE[f(X) | \cF]
\leq
g_{u}(\bE[X | \cF]) \text{ $\bP$-almost surely}.
\end{equation}
Note that, in contrast to \Cref{thm:main_result}, all objects in \eqref{equ:main_result_conditional} are random variables.
\end{theorem}
\begin{proof}
Since $\bE[X | \cF] \in \overline{\CH(K)}$ $\bP$-almost surely and $(\bE[X | \cF],\bE[f(X) | \cF]) \in \overline{\CH(\cG(f))}$ $\bP$-almost surely by \Cref{thm:conditional_mean_in_convex_hull,cor:conditional_mean_in_convex_hull_of_graph},
\[
\bE[f(X) | \cF]
\subseteq
\overline{\CH(\cG(f))}|_{\bE[X | \cF]} 
\subseteq
[g_{l}(\bE[X | \cF]),g_{u}(\bE[X | \cF])] \text{ $\bP$-almost surely}.
\]
\end{proof}

\begin{corollary}
\label{cor:further_conditional_Jensen_type_inequalities}
Let the assumptions of \Cref{thm:main_result_conditional} hold and let $\cY = \bR$.
Then, using the notation from \Cref{cor:further_Jensen_type_inequalities},
\begin{enumerate}[label = (\alph*)]
	\itemsep=2ex	
	\item
	\label{item:conditional_Jensen_type_inequality_2}
	$c_{l}(\bE[X| \cF]) \, f(\bE[X| \cF])
	\leq
	\bE[f(X)]
	\leq
	c_{u}(\bE[X| \cF]) \, f(\bE[X| \cF])$,
	\item
	\label{item:conditional_Jensen_type_inequality_3}
	$
	\max\big(\inf_{x\in K} f(x) \, ,\, f(\bE[X| \cF]) - s_{u}\big)
	\leq
	\bE[f(X)]
	\leq
	\min\big(\sup_{x\in K} f(x) \, ,\, f(\bE[X| \cF]) + s_{l}\big),$
	\item
	\label{item:conditional_Jensen_type_inequality_4}
	$\hat{c}_{l} \, f(\bE[X| \cF])
	\leq
	\bE[f(X)]
	\leq
	\hat{c}_{u} \, f(\bE[X | \cF])$,
\end{enumerate}
\vspace{1ex}
$\bP$-almost surely, whenever these quantities are well defined.
\end{corollary}

\begin{proof}
\ref{item:conditional_Jensen_type_inequality_2} and \ref{item:conditional_Jensen_type_inequality_3} follow directly from \Cref{thm:main_result_conditional}, while \ref{item:conditional_Jensen_type_inequality_4} follows from \ref{item:conditional_Jensen_type_inequality_2}.	
\end{proof}

Again, conditional Jensen's inequality follows almost directly from \Cref{thm:main_result_conditional}:
\begin{corollary}[conditional Jensen's inequality]
	\label{cor:Jensen_follows_from_our_bounds_conditional}	
	Let \Cref{assump:conditional} hold and $f\colon \bR^{d} \to \bR$ be a convex function.
	Then $\bE[f(X)|\cF] \geq f(\bE[X|\cF])$ $\bP$-almost surely.
\end{corollary}
\begin{proof}
	This follows directly from \Cref{thm:main_result_conditional} and \Cref{lemma:for_Jensen}.
\end{proof}


\section*{Acknowledgments}
\addcontentsline{toc}{section}{Acknowledgements}

This research was funded by the Deutsche Forschungsgemeinschaft (DFG, German Research Foundation) under Germany's Excellence Strategy (EXC-2046/1, project 390685689) through the project EF1-10 of the Berlin Mathematics Research Center MATH+.
The author thanks Tim Sullivan for carefully proofreading this manuscript and for helpful suggestions.

\bibliographystyle{abbrvnat}
\bibliography{myBibliography}
\addcontentsline{toc}{section}{References}

\end{document}